\documentclass[12pt]{amsart}
\usepackage[utf8]{inputenc}
\usepackage{amssymb}
\usepackage{amsmath}
\usepackage{mathtools}
\usepackage{mathrsfs}
\usepackage{amsfonts}
\usepackage{enumitem}
\usepackage{tikz}
\usepackage{tikz-cd}
\usepackage[english]{babel}
\usepackage{amsthm}
\usepackage[numbers]{natbib}
\usepackage[justification=centering]{caption}
\usepackage{mathpazo}
\usepackage{yfonts}
\usepackage{graphicx}

\usepackage{geometry}
\DeclarePairedDelimiterX{\inp}[2]{\langle}{\rangle}{#1, #2}

\newtheorem{theorem}{Theorem}[section]

\newtheorem{lemma}[theorem]{Lemma}

\theoremstyle{definition}
\newtheorem{definition}[theorem]{Definition}
\newtheorem*{definition*}{Definition}

\theoremstyle{remark}
\newtheorem{remark}[theorem]{Remark}

\numberwithin{equation}{subsection}

\title{Shortest $k$-Geodesics on Hyperbolic Surfaces}
\author{Changjie Chen}

\begin{document}

\begin{abstract}
    We study the relationship between the lengths of closed geodesics on hyperbolic surfaces and their topological complexity, measured by the self-intersection number. In particular, we provide explicit upper bounds for the length $s_k(X)$ of a shortest closed geodesic with exactly $k$ self-intersections in terms of the length $L_\textswab{8}(X)$ of a shortest figure eight curve, improving Basmajian's estimate. We analyze the geometry of a shortest figure eight curve and explicitly build families of words in $\pi_1(X)$ whose geodesic representatives realize prescribed self-intersection numbers. As a consequence, we improve existing estimates on the maximal self-intersection number $I_k(X)$ of shortest geodesics with at least $k$ self-intersections, reducing the asymptotic upper bound from 512 to 128. This provides a sharper quantitative connection between the geometry and combinatorial complexity of non-simple closed geodesics on hyperbolic surfaces.
\end{abstract}

\maketitle

\section{Introduction}

The geometry of closed geodesics on hyperbolic surfaces has been a central topic in the study of surface geometry and dynamics. A classical problem asks how the length of a closed geodesic is related to its topological complexity, often measured through invariants such as its self-intersection number.

A fundamental object in this setting is the shortest figure eight curve, namely, a closed geodesic with a single geometric self-intersection. Buser \cite{buser2010geometry} proved that any shortest non-simple closed geodesic must be a figure eight curve.

For $k\ge0$, a \emph{$k$-geodesic} on a hyperbolic surface $X$ is a closed geodesic with exactly $k$ geometric self-intersections. Let $s_k(X)$ be the length of a shortest $k$-geodesic on $X$, and let $L_{\textswab{8}}(X)$ be the length of a shortest figure eight curve. Basmajian initiated a comparison between these two quantities. For a compact hyperbolic surface $X$ with (possibly empty) geodesic boundary, he proved in \cite{basmajian2013universal} that
\[
    s_k(X) \le 3(\sqrt{k} + 1) L_{\textswab{8}}(X).
\]
and later improved it in \cite{basmajian21short} to
\[
    s_k(X) \le 2\sqrt{k+\tfrac{1}{4}} L_{\textswab{8}}(X).
\]

In this paper, we further refine the connection between the lengths of non-simple geodesics and the shortest figure eight curves. Our approach is constructive: we work within the pair of pants determined by a shortest figure eight curve, and explicitly build families of words in $\pi_1(X)$ whose geodesic representatives realize prescribed self-intersection numbers.

Our first main result gives an improved upper bound for $s_k(X)$ in terms of $L_\textswab{8}$.
\begin{theorem}[=Theorem\ref{s_k L_8}]
    On a compact hyperbolic surface $X$ with (possibly empty) geodesic boundary, for $k>0$, we have
    \[
    s_k(X)\le \left(\left(k+\frac{1}{4}\right)^\frac{1}{2}+\frac{3}{\sqrt{2}}\left(k+\frac{1}{4}\right)^\frac{1}{4}-\frac{1}{2}\right) L_\textswab{8}(X).
    \]
    In particular,
    \[
    s_k(X) < \left( \sqrt{k}+\frac{3}{\sqrt{2}}\sqrt[4]{k} \right) L_\textswab{8}(X).
    \]
\end{theorem}

As an application, we revisit the quantity $I_k(X)$, the maximal self-intersection number of a shortest closed geodesic on $X$ with at least $k$ self-intersections. Erlandsson and Parlier \cite{erlandsson2020short} prove
\[
I_k(X)\le \left( 32 \sqrt{k+\frac{1}{4}} +1 \right) \left( 16 \sqrt{k+\frac{1}{4}} +1 \right).
\]
They pose the conjecture that for any hyperbolic surface $X$, there is
\[
\lim_{k\to\infty}\frac{I_k(X)}{k}=1,
\]
and prove the conjecture when $X$ has at least one cusp. For closed hyperbolic surfaces, their bound above implies that the limit is less than 512.

Using our improved bound for $s_k(X)$, we obtain a strengthened upper bound for $I_k(X)$ and reduce the limit upper bound from 512 to 128.
\begin{theorem}[=Theorem~\ref{I_k}]
    On a complete hyperbolic surface $X$, we have
    \begin{align*}
        &I_k(X)\le\\
        &\left( 8\left(k+\frac{1}{4}\right)^\frac{1}{2}+12\sqrt{2}\left(k+\frac{1}{4}\right)^\frac{1}{4}-3 \right) \left( 16\left(k+\frac{1}{4}\right)^\frac{1}{2}+24\sqrt{2}\left(k+\frac{1}{4}\right)^\frac{1}{4}-7 \right).
    \end{align*}
    Specifically,
    \[
    I_k(X)< 32\left( \left( 4k+1 \right)^\frac{1}{2} +3 \left( 4k+1 \right)^\frac{1}{4} \right)^2.
    \]
    Thus,
    \[
    \lim_{k\to\infty}\frac{I_k(X)}{k}\le 128.
    \]
\end{theorem}

In Section 2, we describe the local geometry of a figure eight curve and introduce a family of words $w(m,n,j)$ for each prescribed self-intersection number $k$. Section 3 establishes the precise self-intersection number of these words. Section 4 estimates the lengths of the resulting geodesics and proves the main theorems concerning $s_k(X)$ and $I_k(X)$.

\section{Construction of model words}

In this section, we study the local geometry of a shortest figure eight curve on a hyperbolic surface and use it to construct $k$-geodesics.

\begin{definition}
A \emph{simple self-intersection} of a curve $\gamma$ is a point at which $\gamma$ passes through itself exactly twice, i.e.\ the intersection is formed by two distinct local subarcs. The \emph{self-intersection number} $i(\gamma)$ of $\gamma$ on a hyperbolic surface is the minimal number of self-intersections among all curves in the free homotopy class of $\gamma$ with only simple self-intersections.
\end{definition}

\begin{definition}
A \emph{$k$-geodesic} on a hyperbolic surface $X$ is a closed geodesic with self-intersection number $k$. Let $s_k(X)$ denote the length of a shortest $k$-geodesic on $X$, and let $L_\textswab{8}(X)$ denote the length of a shortest figure eight curve.
\end{definition}


On the hyperbolic surface $X$, we choose a shortest figure eight curve $\gamma_\textswab{8}$, which determines an embedded pair of pants $P\subset X$. We will construct $k$-geodesics inside $P$. Let $a,b$ denote the two simple subloops of $\gamma_\textswab{8}$ generating $\pi_1(P)$, as illustrated in Figure~\ref{fig:figure eight}, chosen so that the geodesic representative of the word $ab$ is exactly $\gamma_\textswab{8}$. Without loss of generality we assume
\[
    l(a)\le l(b).
\]

\begin{figure}[ht]
    \centering
    \includegraphics[width=12cm]{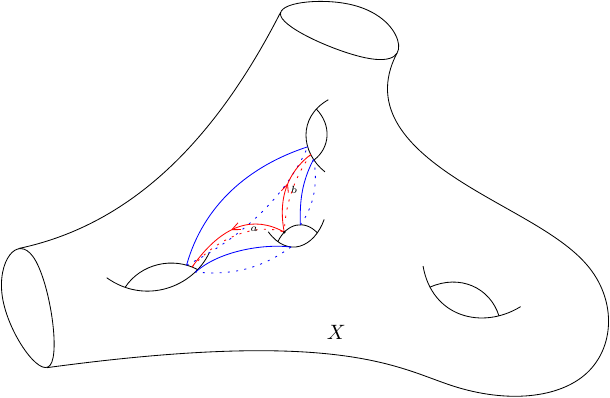}
    \caption{A shortest figure eight curve $\gamma_\textswab{8}$ (in red) determines an embedded pair of pants $P$ (in blue); The loops $a$ and $b$ are the subloops of $\gamma_\textswab{8}$ generating $\pi_1(P)$ and the word $ab$ represents $\gamma_\textswab{8}$.}
    \label{fig:figure eight}
\end{figure}

\begin{definition}
    A \emph{word} in $A=\{a,b,a^{-1},b^{-1}\}$ is a sequence $v_1v_2\dots$ where each $v_i\in A$. A finite word $v_1\dots v_n$ is \emph{cyclically reduced} if $v_i\neq v_{i+1}^{-1}$ for all $1\le i\le n$, where $v_{n+1}=v_1$. The \emph{length} $|w|$ of a cyclically reduced word $w$ is the number of its letters.
\end{definition}

\begin{remark}
    Any cyclically reduced word in $\{a,b,a^{-1},b^{-1}\}$ has the form
    \[
    s_1^{i_1}r_1^{j_1}\dots s_n^{i_n}r_n^{j_n} \text{ or } r_1^{j_1}s_1^{i_1}\dots r_n^{j_n}s_n^{i_n},
    \]
    where $s_i\in\{a,a^{-1}\}$, $r_i\in\{b,b^{-1}\}$, and $i_k,j_k>0$.
\end{remark}

\begin{definition}
    For integers $m,n\ge 0$ and $j>0$, define
    \[
    w(m,n,j)=(ab)^m(a^2b)^nb^j \text{ and } w'(m,n,j)=(ba)^m(b^2a)^na^j,
    \]
    where $m,n\ge 0$ and $j>0$. See Figure~\ref{fig:k geodesic}.
\end{definition}

\begin{figure}[ht]
    \centering
    \includegraphics[width=12cm]{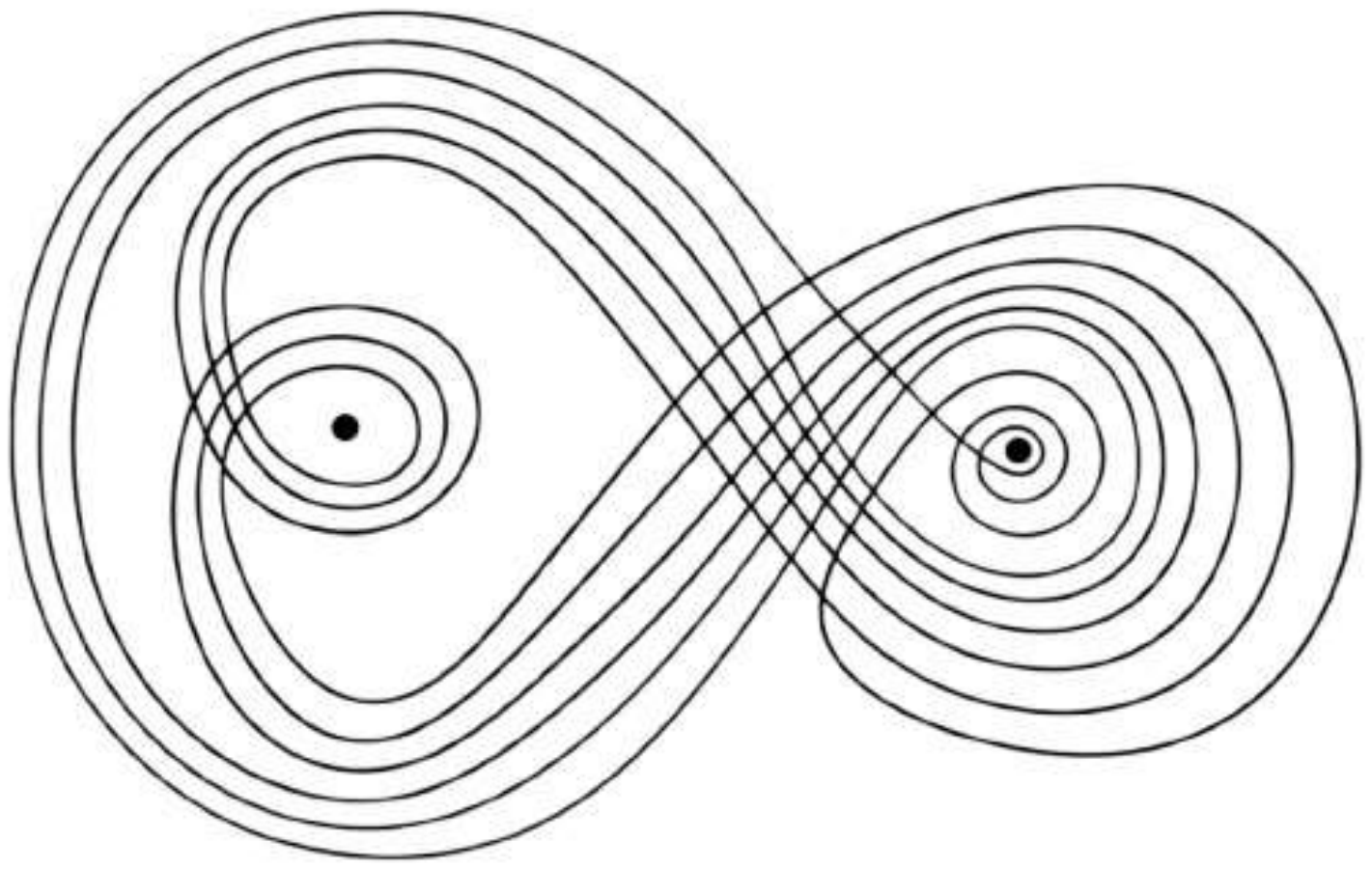}
    \caption{A curve represented by the word $w(3,3,3)$, drawn so that it realizes its self-intersection number, as described in Lemma~\ref{self intersection number of w}}
    \label{fig:k geodesic}
\end{figure}

\section{Calculating the self-intersection number}

The problem of computing the self-intersection number of a word representing a closed curve on a surface has been studied in several papers, including \cite{cohen1987paths, tan1996self, chas2012self}. We recall the setup used in \cite{diop2022self}, which provides a combinatorial formula.

\begin{definition}
    Let
    \[
    w=s_1^{i_1}r_1^{j_1}\dots s_n^{i_n}r_n^{j_n},
    \]
    be a cyclically reduced word with $s_k\in\{a,a^{-1}\}$ and $r_k\in\{b,b^{-1}\}$. Let
    \[
    x_k=s_k^{i_k}r_k^{j_k}\dots s_n^{i_n}r_n^{j_n}\dots s_{k-1}^{i_{k-1}}r_{k-1}^{j_{k-1}},
    \]
    and
    \[
    y_k=r_k^{j_k}s_{k+1}^{i_{k+1}}\dots r_n^{j_n}s_1^{i_1}\dots r_{k-1}^{j_{k-1}}s_k^{i_k}.
    \]
    Fix a fundamental domain for the action of $\pi_1(P)$ on the Poincaré disk $\mathbb D$. Let $\alpha_k$ and $\beta_k$ denote the complete geodesics in $\mathbb D$ joining the limit points of the infinite words $\cdots x_k x_k x_k \cdots$ and $\cdots y_k y_k y_k \cdots$, respectively.
\end{definition}

\begin{definition}
    For $1\le k\le n$, define
    \[
    C_k^1(w)=\{(\alpha_k,\alpha_l): \alpha_k\cap\alpha_l\neq\emptyset, s_k^{i_k}\neq s_l^{i_l}, k<l\le n\},
    \]
    \[
    C_k^2(w)=\{(\beta_k,\alpha_l): \beta_k\cap\alpha_l\neq\emptyset, s_k^{i_k}\neq (s_l^{-1})^{i_l}, k<l\le n\},
    \]
    \[
    D_k^1(w)=\{(\beta_k,\beta_l): \beta_k\cap\beta_l\neq\emptyset, r_k^{j_k}\neq r_l^{j_l}, k<l\le n\}.
    \]
    Further define
    \[
    D_1^2(w)=\{(\alpha_1,\beta_l): \alpha_1\cap\beta_l\neq\emptyset, 1\le l\le n\},
    \]
    and for $2\le k\le n$,
    \[
    D_k^2(w)=\{(\alpha_k,\beta_l): \alpha_k\cap\beta_l\neq\emptyset, r_{k-1}^{j_{k-1}}\neq (r_l^{-1})^{j_l}, k\le l\le n\}.
    \]
    The associated combinatorial quantity is
    \[
    H(w)=\sum_{i=1}^2\sum_{k=1}^n \left(\#C_k^i(w)+\#D_k^i(w)\right).
    \]
\end{definition}

The following theorem gives an explicit relation between the self-intersection number and the function $H(w)$.

\begin{theorem}[\cite{diop2022self}]
\label{self-intersection formula}
    Let $\gamma$ be a closed geodesic on a pair of pants, associated to the cyclically reduced word $w=s_1^{i_1}r_1^{j_1}\dots s_n^{i_n}r_n^{j_n}$, then we have
    \[
    i(\gamma) = H(w) + n|w| - 2n^2 - \sum_{1\le k<l\le n}(|i_k-i_l|+|j_k-j_l|).
    \]
\end{theorem}

\begin{lemma}
\label{self intersection number of w}
    The self-intersection number of
    \[
    w(m,n,j)=(ab)^m(a^2b)^nb^j
    \]
    is
    \[
    i(w(m,n,j)) = (m+n)^2+n^2+m+n+j-1.
    \]
\end{lemma}

\begin{proof}
    We first rewrite $w(m,n,j)$ in the standard form:
    \[
    w(m,n,j)=a_1b_1\dots a_mb_m a_{m+1}^2b_{m+1}\dots a_{m+n-1}^2ba_{m+n}^2b_{m+n}^{j+1}.
    \]
    Thus
    \[
    |w(m,n,j)|=2m+3n+j,
    \]
    and
    \[
    \sum_{1\le k<l\le n}(|i_k-i_l|+|j_k-j_l|)=mn+(m+n-1)j.
    \]
    
    Since $w(m,n,j)$ contains no inverses of $a$ or $b$, the conditions defining $C_k^i$ and $D_k^i$ depend only on comparing exponents. For example,
    \[
    C_k^1(w)=\{(\alpha_k,\alpha_l): \alpha_k\cap\alpha_l\neq\emptyset, i_k\neq i_l, k<l\le n\}.
    \]
    It is easy to check (for example, see Figure~\ref{fig:H function})
    \begin{align*}
        &\#C_k^1=0, \text{ for } 1\le k\le m+n,\\
        &\#D_k^1=1, \text{ for } 1\le k\le m+n-1,\\
        &\#D_{m+n}^1=1,\\
        &\#C_k^2=m+n-k, \text{ for } 1\le k\le m+n,\\
        &\#D_1^2=m+n,\\
        &\#D_k^2=m+n-k+1, \text{ for } 2\le k\le m+n.
    \end{align*}
    
    \begin{figure}[ht]
    \centering
    \includegraphics[width=15cm]{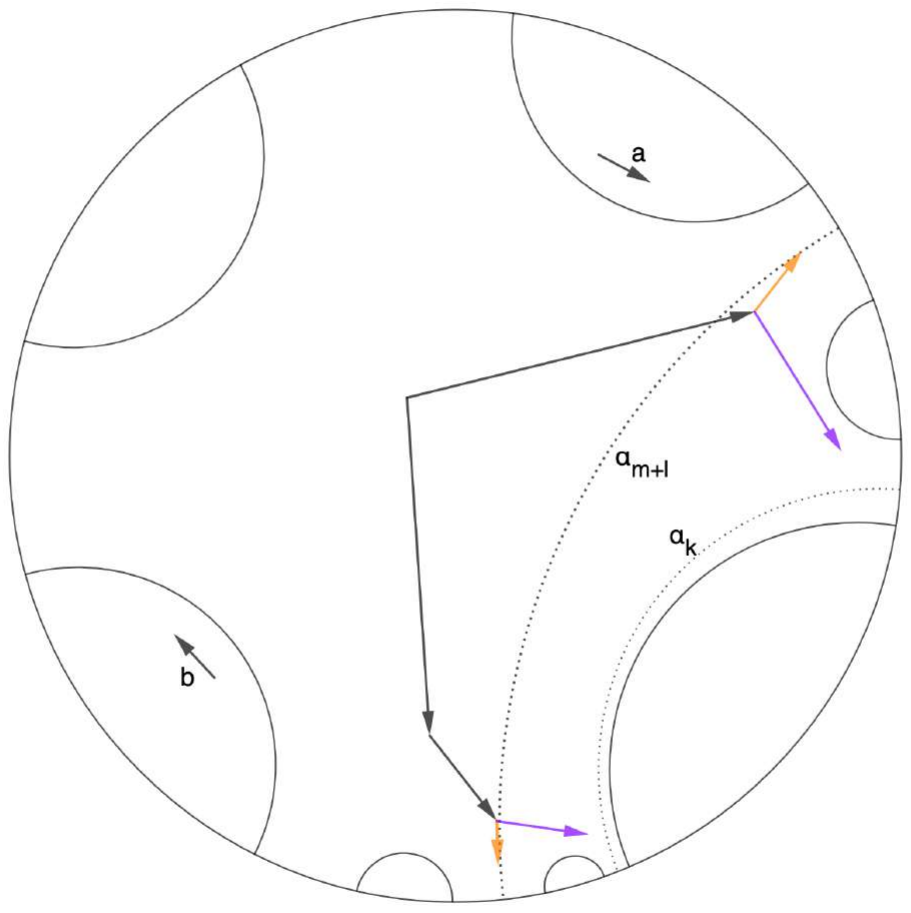}
    \caption{For $1\le k\le m$ and $1\le l\le n$, the infinite words $x_kx_kx_k\cdots$ and $x_{m+l}x_{m+l}\cdots$, as well as their inverses, initially agree along the black arrows before diverging. Hence, $\alpha_k \cap \alpha_{m+l}\neq0$.}
    \label{fig:H function}
    \end{figure}

    Summing these contributions gives
    \[
    H(w)=\sum_{i=1}^2\sum_{k=1}^{m+n}\left(\#C_k^i(w)+\#D_k^i(w)\right)=(m+n)^2+m+n-1,
    \]
    and by Theorem~\ref{self-intersection formula}, we have
    \[
    i(w(m,n,j))=(m+n)^2+n^2+m+n+j-1.
    \]
\end{proof}

\begin{remark}
    Diop--Paris \cite{diop2022self} show that for any word of the form
    \[
    w=a^{i_1}b^{j_1}\dots a^{i_n}b^{j_n},
    \]
    the quantity
    \[
    \sum_{k=1}^n \left(\#C_k^i(w)+\#D_k^i(w)\right)
    \]
    satisfies
    \[
    n-1\le \sum_{k=1}^n \left(\#C_k^i(w)+\#D_k^i(w)\right) \le (n-1)^2.
    \]
    Our family $w(m,n,j)$ realizes this lower bound.
\end{remark}

\section{Length estimates and proofs of the main theorems}

Any integer $k>0$ lies in a unique interval\
\[
[M(M+1),(M+1)(M+2))
\]
for some integer $M\ge 0$. Set
\begin{align*}
a&=k-M(M+1),
&
n&=\lfloor \sqrt{a} \rfloor,
&
m&=M-n,
&
j&=a-n^2+1.
\end{align*}

\begin{lemma}
    For any $k>1$, the parameters $(m,n,j)$ above define a word $w(m,n,j)$, and hence a closed geodesic $\gamma(m,n,j)$ on $P$. If $n\ge j$, we have
    \[
    l(\gamma(m,n,j))\le \left(\left(k+\frac{1}{4}\right)^\frac{1}{2}+\frac{3}{\sqrt{2}}\left(k+\frac{1}{4}\right)^\frac{1}{4}-\frac{1}{2}\right) L_\textswab{8}(X).
    \]
    In particular,
    \[
    l(\gamma(m,n,j))< \left( \sqrt{k}+\frac{3}{\sqrt{2}}\sqrt[4]{k} \right) L_\textswab{8}(X).
    \]
\end{lemma}

\begin{proof}
Clearly $j>0$. Since $k<(M+1)(M+2)$, we have $a\le 2M+1$ and hence $n\le \sqrt{2M+1}$. As $M$ and $n$ are integers, this implies that $M\ge n$ except in the case $(M,n)=(0,1)$, which is ruled out by our assumption that $k>1$.

Note that
\[
\sqrt{a}-1 < \lfloor \sqrt{a} \rfloor \le \sqrt{a},
\]
which implies
\[
j=a-n^2+1<2\sqrt{a}\le2\sqrt{2M+1}.
\]

When $n\ge j$, we have
\begin{align*}
    l(\gamma(m,n,j)) & \le l(w(m,n,j))\\
    &= (m+2n)l(a)+(m+n+j)l(b)\\
    &= (m+n+j)l(ab)+(n-j)l(a)\\
    &\le (m+n+j)L_\textswab{8}+\frac{n-j}{2}L_\textswab{8}\\
    &=\frac{2m+3n+j}{2}L_\textswab{8},
\end{align*}
where
\begin{align*}
    \frac{2m+3n+j}{2}&=M+\frac{n+j}{2}\\
    &\le M+\frac{3}{2}\sqrt{M+\frac{1}{2}}\\
    &\le \left(k+\frac{1}{4}\right)^\frac{1}{2}+\frac{3}{\sqrt{2}}\left(k+\frac{1}{4}\right)^\frac{1}{4}-\frac{1}{2}.
\end{align*}
Finally, it is easy to find
\[
\left(k+\frac{1}{4}\right)^\frac{1}{2}+\frac{3}{\sqrt{2}}\left(k+\frac{1}{4}\right)^\frac{1}{4}-\frac{1}{2}<\sqrt{k}+\frac{3}{\sqrt{2}}\sqrt[4]{k},
\]
which proves the lemma.
\end{proof}

\begin{theorem}
\label{s_k L_8}
    On a compact hyperbolic surface $X$ with (possibly empty) geodesic boundary, for $k>0$, we have
    \[
    s_k(X) \le \left(\left(k+\frac{1}{4}\right)^\frac{1}{2}+\frac{3}{\sqrt{2}}\left(k+\frac{1}{4}\right)^\frac{1}{4}-\frac{1}{2}\right) L_\textswab{8}(X).
    \]
    In particular,
    \[
    s_k(X)< \left( \sqrt{k}+\frac{3}{\sqrt{2}}\sqrt[4]{k} \right) L_\textswab{8}(X).
    \]
\end{theorem}

\begin{proof}
    The case $k=1$ is trivial. For $k>1$, choose $(m,n,j)$ as above.
    
    If $n\ge j$, the length of the geodesic $\gamma(m,n,j)$ gives the claimed upper bound.
    
    If $n<j$, the word
    \[
    w'(m,n,j)=(ba)^m(b^2a)^na^j
    \]
    defines a geodesic $\gamma'(m,n,j)$, whose length gives the same upper bound.
\end{proof}

Let $I_k(X)$ denote the maximal self-intersection number of a shortest geodesic on $X$ with at least $k$ self-intersections. Erlandsson and Parlier \cite{erlandsson2020short} proved the following.

\begin{theorem}[\cite{erlandsson2020short}]
\label{EP}
    Let $X$ be an orientable complete hyperbolic surface with non-abelian fundamental group. Then
    \[
    I_k(X)\le \left( 32 \sqrt{k+\frac{1}{4}} +1 \right) \left( 16 \sqrt{k+\frac{1}{4}} +1 \right).
    \]
\end{theorem}

Using Theorem~\ref{s_k L_8}, we improve the linear coefficient from $512$ to $128$.

\begin{theorem}
\label{I_k}
    On a complete hyperbolic surface $X$, we have
    \begin{align*}
        &I_k(X)\le\\
        &\left( 8\left(k+\frac{1}{4}\right)^\frac{1}{2}+12\sqrt{2}\left(k+\frac{1}{4}\right)^\frac{1}{4}-3 \right) \left( 16\left(k+\frac{1}{4}\right)^\frac{1}{2}+24\sqrt{2}\left(k+\frac{1}{4}\right)^\frac{1}{4}-7 \right).
    \end{align*}
    In particular,
    \[
    I_k(X)< 32\left( \left( 4k+1 \right)^\frac{1}{2} +3 \left( 4k+1 \right)^\frac{1}{4} \right)^2.
    \]
    Thus,
    \[
    \lim_{k\to\infty}\frac{I_k(X)}{k}\le 128.
    \]
\end{theorem}

\begin{proof}
    In the proof of Theorem~\ref{EP}, Erlandsson and Parlier consider a partition of a shortest closed geodesic with at least $k$ self-intersections into $m$ segments, each of length at most $\frac{1}{8}L_\textswab{8}$. They obtain the following bounds
    \[
    m<8\frac{s_k}{L_\textswab{8}}+1,
    \]
    and
    \[
    I_k(X)\le 2m^2-m.
    \]
    Substituting our improved upper bound for $\frac{s_k}{L_\textswab{8}}$ from Theorem~\ref{s_k L_8} yields the claimed estimate.
\end{proof}


\clearpage
\bibliographystyle{alpha}
\bibliography{main}

\end{document}